\documentclass[11pt]{amsart}
\usepackage{amssymb}
\usepackage{latexsym}
\usepackage{tikz}

\newtheorem{theorem}{Theorem}[section]
\newtheorem{proposition}[theorem]{Proposition}
\newtheorem{corollary}[theorem]{Corollary}
\newtheorem{lemma}[theorem]{Lemma}

\theoremstyle{definition}

\newtheorem{definition}[theorem]{Definition}

\newtheorem{remark}[theorem]{Remark}

\newtheorem{problem}[theorem]{Problem}

\newcommand{\C}{\mathbb{C}}

\newcommand{\D}{\mathbb{D}}
\newcommand{\N}{\mathbb{N}}
\newcommand{\Z}{\mathbb{Z}}

\newcommand{\abs}[1]{\left| #1 \right|}
\DeclareMathOperator{\h}{h}
\let\b\relax
\DeclareMathOperator{\b}{b}

\title[Hardy number and Bergman number]{On the Hardy number and the Bergman number of a planar domain}
\author[D. Betsakos]{Dimitrios Betsakos}
\address{Department of Mathematics, Aristotle University of Thessaloniki, 54124, Thessaloniki, Greece}
\email{betsakos@math.auth.gr}
\author[F. J. Cruz-Zamorano]{Francisco J. Cruz-Zamorano}
\address{Departamento de Matem\'{a}tica Aplicada II and IMUS, Escuela T\'{e}cnica Superior de Ingenier\'{i}a, Universidad de Sevilla, Camino de los Descubrimientos, S/N 41092, Sevilla, Spain}
\email{fcruz4@us.es}

\date{\today}
\thanks{Cruz-Zamorano was supported by Ministerio de Innovaci\'on y Ciencia, Spain, project PID2022-136320NB-I00 and by Ministerio de Universidades, Spain, through the action Ayuda del Programa de Formaci\'on de Profesorado Universitario, reference FPU21/00258. This work was initiated when Cruz-Zamorano visited the Aristotle University of Thessaloniki, funded by Ministerio de Ciencia, Innovaci\'on y Universidades, reference EST24/00450.}
\keywords{Hardy number, Bergman number, Greenian domain, Green function, harmonic measure, hyperbolic metric.}
\subjclass[2020]{Primary: 30H10, 30H20; Secondary: 31A15.}

\numberwithin{equation}{section}

\allowdisplaybreaks

\begin{document}

\begin{abstract}
This article deals with functions with a prefixed range and their inclusions in Hardy and weighted Bergman spaces. This idea was originally introduced by Hansen for Hardy spaces, and it was recently taken into weighted Bergman spaces by Karafyllia and Karamanlis. We provide a new characterization for the Hardy number of a domain in terms of its Green function. Based on this, we present a class of domains for which the Hardy number and the Bergman number coincide. However, in general, we show that the Hardy number and the Bergman number of a domain are not equal; even for domains which are regular for the Dirichlet problem.
\end{abstract}

\maketitle

\section{Introduction}
This article is devoted to analyze the Hardy and weighted Bergman spaces to which holomorphic functions with a prefixed range belong. To introduce these ideas, let us recall that the Hardy space $H^p$ for $0 < p < +\infty$ is defined as the space of holomorphic functions $f$ defined on the unit disk $\D$ for which
$$\sup_{0 < r < 1}\int_0^{2\pi}|f(re^{i\theta})|^pd\theta < +\infty.$$
In the case $p = +\infty$, $H^{\infty}$ stands for the space of bounded holomorphic functions on $\D$. An introduction to these spaces can be found in \cite{Duren-Hp}. 

It is known \cite[p. 2]{Duren-Hp} that $H^q \subset H^p$ whenever $0 < p \leq q \leq +\infty$. The Hardy number of a function $f$ that is holomorphic on $\D$ is defined as
$$\h(f) = \sup(\{0\} \cup \{p > 0 : f \in H^p\}) \in [0,+\infty].$$
In \cite{Hansen1}, Hansen introduced the Hardy number of a domain $D \subset \C$, which is defined as
\begin{align*}
\h(D) & = \sup(\{0\} \cup \{p > 0 : \mathrm{Hol}(\D,D) \subset H^p\}) \\
& = \inf\{\h(f) : f \in \mathrm{Hol}(\D,D)\} \in [0,+\infty].   
\end{align*}
Here, $\mathrm{Hol}(\D,D)$ denotes the set of all holomorphic functions $f \colon \D \to D$.
There are some recent works in which this topic is studied. Some of them (see \cite{Essen,K-HypDist,KimSugawa}) estimate the Hardy number of general domains using the harmonic measure or the hyperbolic distance. Some others (see \cite{CCZKRP,K-Comb2}) analyze the Hardy number of domains with specific geometric properties.

In \cite{Essen}, Ess\'en related the Hardy number of a domain with certain harmonic measures. Later, Kim and Sugawa \cite{KimSugawa} provided a formula for the Hardy number of a domain using Ess\'en's theorem; see Subsection \ref{subsec:EKS}. The first result of this article deals with the Hardy number of a domain and its Green function. To introduce it, we recall some tools: If the complement (in $\C$) of a domain $D\subset \C$ contains at least two points, then $D$ is a {\it hyperbolic domain} and there exists a universal covering map $f_D \colon \D \to D$; see \cite[Chapter 15]{Marshall}. If $D$ is a domain and its complement is non-polar, then $D$ possesses a  Green function; see Subsection \ref{subsec:green}. Such domains are called {\it Greenian}. In particular, Greenian domains are also hyperbolic. Without loss of generality, we will sometimes assume that $0 \in D$ and that the universal covering map $f_D$ is such that $f_D(0) = 0$. In such a case, if $g_D$ is the Green function of $D$, we extend it so that $\C \setminus \{0\} \ni z \mapsto g_D(z,0)$ is defined to be $0$ whenever $z \not\in D$. Under the latter setting, we define the function
\begin{equation}
\label{eq:integralmean}
\psi_D(r) = \int_0^{2\pi}g_D(re^{i\theta},0)d\theta, \quad 0 < r < +\infty.
\end{equation}

\begin{theorem}
\label{thm:main-formula-Hardy}
Let $D \subset \C$ be a Greenian domain. Assume that $0 \in D$, and let $f_D \colon \D \to D$ be a universal covering map with $f_D(0) = 0$. Let $\psi_D$ be as in \eqref{eq:integralmean}. Then,
\begin{equation*}
\h(D) = \h(f_D) = \liminf_{r \to +\infty}\left(-\dfrac{\log \psi_D(r)}{\log r}\right).
\end{equation*}
\end{theorem}
We recall that if a domain $D \subset \C$ is not Greenian, then $\h(D) = 0$; see Subsection \ref{subsec:EKS}.

The function $\psi_D$ in \eqref{eq:integralmean} can also be used to connect the Hardy number and the Bergman number of a domain. To present this, we recall that the weighted Bergman space $A^p_{\alpha}$ for $0 < p < +\infty$ and $\alpha > -1$ is the space of holomorphic functions $f$  on $\D$ for which
$$\int_{\D}\abs{f(z)}^p(1-\abs{z})^{\alpha}dA(z) < +\infty,$$
where $A$ denotes the (Lebesgue) area measure. An exposition regarding these spaces can be found in \cite{Duren-Apa}.

In \cite{BKK}, Betsakos, Karafyllia, and Karamanlis analyzed the weighted Bergman spaces to which a given conformal map belongs. Later, Karafyllia and Karamanlis \cite{KK} introduced the Bergman number of a holomorphic map $f$ on $\D$, that is,
$$\b(f) = \sup\left(\{0\} \cup \left\lbrace \dfrac{p}{\alpha+2} : p > 0, \, \alpha > -1, \, f \in A^p_{\alpha}\right\rbrace\right) \in [0,+\infty].$$
They showed that $\h(f) = \b(f)$ for all conformal maps on $\D$ in contrast to the general case; see Remark \ref{remark:hfzero}.

In \cite{K-Domain}, Karafyllia also introduced the Bergman number of a domain $D \subset \C$, that is,
\begin{equation}\label{b-def}
\b(D) = \inf\{\b(f) : f \in \mathrm{Hol}(\D,D)\} \in [0,+\infty].
\end{equation}
To continue with these ideas, we introduce Bergman numbers depending on the weight: For $\alpha > -1$ and a holomorphic map $f$ on $\D$, set
$$\b_{\alpha}(f) = \sup(\{0\} \cup \{p > 0: f \in A^p_{\alpha}\}) \in [0,+\infty].$$
Similarly, for a domain $D \subset \C$, we also define
\begin{align}
\label{eq:def-balpha}
\b_{\alpha}(D) & = \sup(\{0\} \cup \{p > 0 : \mathrm{Hol}(\D,D) \subset A^p_{\alpha}\}) \\
& = \inf\{\b_{\alpha}(f) : f \in \mathrm{Hol}(\D,D)\} \in [0,+\infty]. \nonumber
\end{align}
It follows from the work of Karafyllia and Karamanlis \cite{KK} that $\b(D) = \h(D)$ for every simply connected domain $D\subset \C$. In the present article, we further investigate about this equality for multiply connected domains. For instance, we construct some domains for which the Hardy number and the Bergman number differ. In particular, one of these examples disproves a previous claim in \cite[Theorem 1.1]{K-Domain}.
\begin{theorem}
\label{thm:Hardy-neq-Bergman}
{\rm (a)} For every $p\in [0,+\infty)$ there exists a domain $D \subset \C$ such that $\h(D) = p$ and $\b(D) = \b_{\alpha}(D) = +\infty$, $\alpha > -1$. \\
{\rm (b)} There exists a domain $D \subset \C$ which is regular for the Dirichlet problem but $\h(D)=1/2$ and $\b(D)=\b_{\alpha}(D) = +\infty$, $\alpha > -1$.
\end{theorem}

In contrast, we also show that the equality between the Hardy number and the Bergman number holds for a class of multiply connected domains. This class is given as follows:
\begin{definition}
\label{def:classD}
A domain $D\subset \C$ belongs to the class $\mathcal{D}$, if it has the following properties:
\\
(a) $D$ is unbounded.\\
(b) Let $F$ be the union of all bounded components of $\C_\infty\setminus D$. The set $F$ is bounded.\\
(c) Consider the simply connected domain $\Omega=D\cup F$. For all sufficiently large $r>0$, the set $\Omega\cap \{z\in \C:|z|=r\}$ has exactly one component. 
\end{definition}
A half-strip or a sector having a finite number of holes are typical domains in the class $\mathcal{D}$. Also, $\C\setminus \{z_1,z_2,\dots,z_n\}$ is a domain in $\mathcal{D}$. 
\begin{theorem}
\label{thm:class-D}
Let $\alpha > -1$, $D\in\mathcal{D}$ and let $\Omega$ be the simply connected domain associated with $D$ in Definition \ref{def:classD}. Then
$$\h(D) =\b(D)=\h(\Omega)=\b(\Omega)=\dfrac{\b_{\alpha}(D)}{\alpha+2}=\dfrac{\b_{\alpha}(\Omega)}{\alpha+2}.$$
If, in addition, $D$ is a hyperbolic domain with $0 \in D$, we have
$$\h(f_D) = \b(f_D) = \dfrac{\b_{\alpha}(f)}{\alpha+2},$$
where $f_D \colon \D \to D$ is a universal covering map with $f_D(0) = 0$.
\end{theorem}

\medskip

The paper is organized as follows: In Section \ref{sec:preliminaries} we provide some preliminaries needed throughout the paper. Next, in the subsequent sections, we prove Theorems \ref{thm:main-formula-Hardy} and \ref{thm:Hardy-neq-Bergman}. In Section \ref{sec:class-D-lemmas} we provide some properties of the class $\mathcal{D}$. Finally, in Section \ref{sec:proof-D}, we present the proof of Theorem \ref{thm:class-D}. We end with some open problems in Section \ref{sec:open-problems}.

\bigskip

\textbf{Acknowledgments.}  We thank Manuel D. Contreras, Christina Karafyllia, and Luis Rodr\'iguez-Piazza for their remarks, corrections, and comments. We also thank Adri\'an Llinares for suggesting reference \cite{Arevalo}.


\section{Preliminaries}
\label{sec:preliminaries}

\subsection{Harmonic measure}
 Let $\Omega$ be a domain in $\mathbb{C}$  and let $E$ be a Borel subset of $\partial\Omega$. The harmonic measure of $E$ with respect to $\Omega$ is the solution of the generalized Dirichlet problem for the Laplacian in $\Omega$ with boundary function $\chi_E$; see \cite[Theorem 4.2.6]{Ransford}. For $z\in\Omega$, we are going to use the standard notation $\omega(z,E,\Omega)$ for the harmonic measure of $E$ with respect to $\Omega$, evaluated at $z\in\Omega$.  We will use the same notation even if $\Omega$ is only an open set; in this case, we consider the component of $\Omega$ containing the point $z$. 
 
 Note that the harmonic measure, as a function of $E$, is a Borel probability measure on $\partial \Omega$.
 The basic properties of harmonic measure are the maximum principle and the conformal invariance. We briefly review now  some additional properties of harmonic measure that we will need. 

The first property we state is the \textit{Strong Markov Property}. Let $\Omega_1\subset\Omega_2$ be domains and $E$ be a Borel subset of $\partial \Omega_1\cap\partial \Omega_2$. Then, for all $z\in\Omega_1$, by \cite[p. 88]{PS}, 
\begin{equation}
\label{markov property}
\omega(z,E,\Omega_2)=\omega(z,E,\Omega_1)+\int\limits_{\partial\Omega_1\cap\Omega_2}\omega(\zeta,E,\Omega_2)\;\omega(z,d\zeta,\Omega_1).
\end{equation}

We will also need a polarization result of Solynin. We state here a simple version of \cite[Theorem 2]{Sol}. Let $\Omega$ be a planar domain and let $\ell$ be a straight line dividing the plane into two open half-planes $H^+,H^-$. Set $\Omega^+=\Omega\cap H^+$ and $\Omega^-=\Omega\cap H^-$. Let $*$ denote the reflection in $\ell$. 
\begin{theorem}
\cite[Theorem 2]{Sol}
\label{polarization}
Let $E$ be a Borel subset of $\partial \Omega$. Assume that $(\Omega^-)^*\subset \Omega^+$.\\
{\rm (a)} If $z\in \ell\cap\Omega$, $E\subset \partial \Omega\cap\partial \Omega^-$,
and $E^*\subset \partial \Omega$, then
$\omega(z,E,\Omega)\leq\omega(z,E^*,\Omega)$.\\
{\rm (b)} If $z\in \Omega^-$ and $E=E^*$, then $\omega(z,E,\Omega)\leq\omega(z^*,E,\Omega)$.
\end{theorem}

\subsection{Green function and hyperbolic metric}
\label{subsec:green}
A Borel set $E \subset \C$ is called polar if
$$\int_{\C}\int_{\C}\log\abs{z-w}d\mu(z)d\mu(w) = -\infty$$
for all Borel probabilities $\mu$ whose support is a compact set lying on $E$. A property is said to hold for {\it nearly every} point  of a set $A \subset \C$ (briefly,  {\it n.e.}) if there exists a polar subset $E \subset A$ such that the property holds on $A \setminus E$.
For more details on this concept we refer to \cite[Sections 3.2 and 5.1]{Ransford}.

\medskip

A Green function for a domain $D \subset \C$ is a map $g_D \colon D \times D \to (-\infty,+\infty]$ such that, for all $w \in D$, the following properties hold:\\
{\rm (i)} The function $z \mapsto g_D(z,w)$ is harmonic on $D \setminus \{w\}$ and bounded outside every neighbourhood of $w$.\\
{\rm (ii)} $g_D(w,w) = +\infty$; moreover, $g(z,w) = -\log\abs{z-w}+\mathrm{O}(1)$ as $z \to w$.\\
{\rm (iii)} For nearly every $\xi \in \partial D$, $g_D(z,w) \to 0$ as $z \to \xi$.

\medskip

A domain $D \subset \C$ possesses a (necessarily unique) Green function if and only if $\C \setminus \D$ is not a polar set; see \cite[Theorem 4.4.2]{Ransford}. As stated before, such domains are called Greenian. In particular, if $D$ is a Greenian domain, then $g_D(z,w) > 0$ for all $z,w \in D$; see \cite[Theorem 4.4.3]{Ransford}. Thus, the function $\psi_D$ defined on \eqref{eq:integralmean} is non-negative. In the case of the unit disk $\D$, its Green function is given by
\begin{equation*}
g_\D(0,z)=\log\frac{1}{|z|},\;\;\;z\in \D.
\end{equation*}
We will use the following relation with universal covering maps:
\begin{theorem}
	{\normalfont(see e.g. \cite[Theorem 15.6]{Marshall})}
	\label{thm:green}
	Let $D \subset \C$ be a domain with Green function $g_D$ and universal covering map $f_D \colon \D \to D$. Then, for all $a \in \D$ and all $w \in D$, $w \neq f(a)$, it holds
	$$g_D(f(a),w) = \sum_{j}g_{\D}(a,z_j(w)),$$
	where $\{z_j(w)\} \subset \D$ are all the preimages of $w$ by $f_D$.
\end{theorem}

Another basic property of the Green function is its domain monotonicity: If $D,\Omega$ are Greenian domains and $D\subset \Omega$, then 
\begin{equation*}
g_D(z,w)\leq g_\Omega(z,w),\;\;\;z,w\in D.
\end{equation*}
A more precise relation between $g_D$ and $g_\Omega$ is established by the Strong Markov Property (see \cite[p. 54]{PS}): For $z,w\in D\subset \Omega$,
\begin{equation}
\label{smp}
g_\Omega(z,w)=g_{D}(z,w)+\int_{\partial D\cap\Omega}g_\Omega(\zeta,w)\omega(z,d\zeta, D).
\end{equation}

\medskip

We will need some basic properties of the hyperbolic metric. We refer to \cite{BCD} and \cite{Hay} for detailed accounts. First we recall the definition of the hyperbolic distance for the unit disk. 
If $z_1,z_2\in \D$, then 
$$
\rho_\D(z_1,z_2)=\frac{1}{2}\log\frac{1+\sigma(z_1,z_2)}{1-\sigma(z_1,z_2)},\;\;\;\sigma(z_1,z_2)=\left |\frac{z_1-z_2}{1-\overline{z_1}z_2}\right |.
$$

For a hyperbolic domain $D \subset \C$, the hyperbolic distance $\rho_D$ is defined as follows: Let $\zeta,w\in D$. Consider a universal covering map (for $D$) $f_D:\D\to D$ with $f_D(0)=\zeta$. Let $z_j(w), j=1,2,\dots$, be the preimages of $w$ under $f_D$ with $|z_1(w)|\leq |z_2(w)|\leq\dots$. Then 
\begin{equation}
\label{hd}
\rho_D(\zeta,w):=\rho_\D(0,z_1(w)) = \min_{\substack{z \in \D \\ f(z) = w}}\rho_{\D}(0,z).
\end{equation}
  
The following lemma presents some relations between Green function and hyperbolic distance. 

\begin{lemma}\label{PL1}
	{\rm (a)} Let $D\subset \C$ be a simply connected domain. Then for $z,w\in D,\; z\neq w$,
	\begin{equation}\label{PL1e1}
	g_D(z,w)=\log\frac{1+e^{-2\rho_D(z,w)}}{1-e^{-2\rho_D(z,w)}}.
	\end{equation}
	{\rm (b)} 	Let $D\subset \C$ be a simply connected domain. If $z,w\in D$ and $\rho_D(z,w)\geq \frac{\log 2}{2}$, then
	\begin{equation}
    \label{PL1e2}
	g_D(z,w)\leq 4\,e^{-2\rho_D(z,w)}.
	\end{equation}
	{\rm (c)}	If $D\subset \C$ is a Greenian domain, then for $z,w\in D,\; z\neq w$,
	\begin{equation}
    \label{PL1e3}
	g_D(z,w)\geq e^{-2\rho_D(z,w)}.
	\end{equation}	
\end{lemma}	
\begin{proof}
	(a) If $D=\D$ and $z=0$, then (\ref{PL1e1}) is an easy consequence of the explicit formulas for the Green function and the hyperbolic distance on $\D$. The general case follows from conformal invariance.\\
	(b) Because of (\ref{PL1e1}) and the inequality $\rho_D(z,w)\geq \frac{\log 2}{2}$, we have
	\begin{eqnarray}
    \label{PL1p1}
	g_D(z,w)&=&\log\frac{1+e^{-2\rho_D(z,w)}}{1-e^{-2\rho_D(z,w)}}=\log\left (1+\frac{2e^{-2\rho_D(z,w)}}{1-e^{-2\rho_D(z,w)}}\right ) \\&\leq & \frac{2e^{-2\rho_D(z,w)}}{1-e^{-2\rho_D(z,w)}}\leq 4\,e^{-2\rho_D(z,w)}.\nonumber
	\end{eqnarray}
	(c) Let $z,w\in D,\; z\neq w$. Consider a universal covering map $f_D \colon \D \to D$ with $f_D(0)=z$. Let $z_j(w)$, $j=1,2,\dots$, be the preimages of $w$ under $f_D$ with $\abs{z_1(w)}\leq \abs{z_2(w)}\leq \dots$.  By (\ref{thm:green}) and  (\ref{PL1e1}),
	\begin{eqnarray}
    \label{PL1p2}
	g_D(z,w)& =& \sum_{j}g_{\D}(0,z_j(w))\geq g_\D(0,z_1(w))\\&=&\log\frac{1+e^{-2\rho_\D(0,z_1(w))}}{1-e^{-2\rho_\D(0,z_1(w))}}.\nonumber
	\end{eqnarray}
	Using the elementary inequality $\log\frac{1+x}{1-x}\geq x$, $0<x<1$, and (\ref{hd}), we infer from (\ref{PL1p1}) that
	\begin{eqnarray}
    \label{PL1p3}
	g_D(z,w)\geq e^{-2\rho_\D(0,z_1(w))}=e^{-2\rho_D(z,w)}.
	\end{eqnarray}
\end{proof}


\subsection{Littlewood-Paley formulas.}
In order to relate the Hardy and Bergman number of a domain with its Green function, the following classical characterizations will be used.
\begin{theorem}
	\cite[Theorem 1]{Yamashita}
	\label{thm:Hp}
	Let $f \colon \D \to \C$ be a holomorphic function. For $0 < p < +\infty$, $f \in H^p$ if and only if
	$$\int_{\D}\abs{f(z)}^{p-2}\abs{f'(z)}^2\log(1/\abs{z})dA(z) < +\infty.$$
\end{theorem}
\begin{remark}
	The integrability condition of the latter result is not the one stated in \cite[Theorem 1]{Yamashita}, but it is easy to see that they are equivalent.
\end{remark}

\begin{theorem}
	{\normalfont \cite[Lemma 2.3]{Smith}}
	\label{thm:Apa}
	Let $f \colon \D \to \C$ be a holomorphic function. For $0 < p < +\infty$ and $\alpha > -1$, $f \in A^p_{\alpha}$ if and only if
	$$\int_{\D}\abs{f(z)}^{p-2}\abs{f'(z)}^2\log(1/\abs{z})^{\alpha+2}dA(z) < +\infty.$$
\end{theorem}


\subsection{The Ess\'en-Kim-Sugawa formula.}
\label{subsec:EKS}
Kim and Sugawa \cite[Lemma 3.2]{KimSugawa}, based on an older theorem of Ess\'en \cite[Lemma 1]{Essen},
proved the following result: Let $D$ be a domain containing the origin. For $R>0$,
let $\Delta_R=\{z\in\C: |z|<R\}$ and $D_R$ be the connected component of $D \cap \Delta_R$ containing $0$.  Set
$\omega(R)=\omega(0, \partial D_R \cap \partial\Delta_R, D_R)$. Then
\begin{equation}
\label{KS}
\h(D)= \liminf_{R\to +\infty}\left(-\frac{\log \omega(R)}{\log R}\right).
\end{equation}

In particular, if a domain $D \subset \C$ is not Greenian (i.e., $\C \setminus D$ is polar), then $\h(D) = \h(\C) = 0$; see \cite[Theorem 2.4.(i)]{KimSugawa}. This also follows from an older result due to Nevanlinna.

\subsection{Inclusions among Hardy and weighted Bergman spaces.}
Inclusions between Hardy spaces are simple and very well known \cite[p. 2]{Duren-Hp}: $H^q \subset H^p$ if and only if $0 < p \leq q \leq +\infty$. This is, indeed, what motivates the definition of the Hardy number. There are also classical results relating Hardy and Bergman spaces. For example:
\begin{theorem}
	\label{thm:HqsubsetApa}
	{\normalfont (see \cite[Theorem 5.11]{Duren-Hp})}
	Let $p,q > 0$ and $\alpha > -1$. If $q \geq p/(\alpha+2)$, then $H^q \subset A^p_{\alpha}$.  In particular, $H^p \subset A^p_{\alpha}$ for all $p > 0$, $\alpha > -1$.
\end{theorem}
However, the reversed inclusions completely fail; see \cite[p. 80]{Duren-Apa} for a discussion and \cite[Theorem 5.10]{Duren-Hp} for a wider setting:
\begin{theorem}
	\label{thm:ApanotsubsetN}

	Let $p > 0$, $\alpha > -1$. The space $A^p_{\alpha}$ contains functions that are not in the Nevanlinna class, and therefore they are not in any Hardy space.
\end{theorem}
For a definition of the Nevanlinna class, its properties, and the inclusion of Hardy spaces in it, we refer to \cite[Section 2.1]{Duren-Hp}.

The inclusions between mixed-norm spaces have been completely characterized in \cite{Arevalo}. In particular, this yields a similar characterization for the case of Bergman spaces. Indeed, the following result is a consequence of \cite[Theorems 1 and 2]{Arevalo}:
\begin{theorem}
	\label{thm:Arevalo}
	Consider $p,q > 0$ and $\alpha,\beta > -1$.\\
{\rm (a)} Suppose that $p = q$. Then, $A^p_{\alpha} \subset A^q_{\beta}$ if and only if $\alpha \leq \beta$.
\\
{\rm (b)} Suppose that $p > q$. Then, $A^p_{\alpha} \subset A^q_{\beta}$ if and only if
		$$\dfrac{\alpha+1}{p} < \dfrac{\beta+1}{q}.$$
\\
{\rm (c)} Suppose that $p < q$. Then, $A^p_{\alpha} \subset A^q_{\beta}$ if and only if
		$$\dfrac{\alpha+2}{p} \leq \dfrac{\beta+2}{q}.$$
\end{theorem}

\subsection{Initial remarks on Hardy and Bergman numbers.}
The inclusions in the latter subsection impose several relations among the Hardy number $\h$ and the Bergman numbers $\b$ and $\b_{\alpha}$. Some of them have been previously noticed in \cite[Lemmas 2.1 and 2.2]{K-Domain}:
\begin{lemma}
	\label{lemma:ineq-b}
		{\rm (a)}  Let $f$ be a holomorphic map on $\D$. Then $\h(f) \leq \b(f)$.\\
		{\rm (b)} Let $D \subset \C$ be a domain. Then $\h(D) \leq \b(D)$.\\
		{\rm (c)} If $\C \setminus D$ is bounded, then there exists a holomorphic function $f \colon \D \to D$ such that $f \not\in A^p_{\alpha}$ for all $p > 0$, $\alpha > -1$. In particular, $\b(D) = 0$.

\end{lemma}
\begin{remark}
	\label{remark:hfzero}
	It follows from Theorem \ref{thm:ApanotsubsetN} that the reverse inequality in Lemma \ref{lemma:ineq-b}(a) fails: if $f \in A^p_{\alpha}$ is such that $f \not\in H^q$ for all $q > 0$, then $\b(f) > 0$ and $\b_{\alpha}(f) > 0$, but $\h(f) = 0$. However, as already noted, $\h(f) = \b(f)$ for a conformal map $f$; see \cite[Eq. (1.5)]{KK}.
\end{remark}

In a similar fashion, we can also give the following inequality:
\begin{lemma}
	\label{lemma:ineq-ba}
	Let $D \subset \C$ be a domain. For all $\alpha > -1$, 
	$$
	\h(D) \leq \frac{\b_{\alpha}(D)}{\alpha+2}.
	$$
	\begin{proof}
		If $\h(D) = 0$, the inequality holds. Therefore, let us assume that $\h(D) > 0$. In that case, choose $0 < q < \h(D)$ and a function $f \colon \D \to D$. By the definition of the Hardy number, $f \in H^q$. But then it follows from Theorem \ref{thm:HqsubsetApa} that $f \in A^p_{\alpha}$ whenever $q \geq p/(\alpha+2)$. This implies that $q \leq \b_{\alpha}(D)/(\alpha+2)$. The result follows by taking the limit as $q \to \h(D)$.
	\end{proof}
\end{lemma}

The Hardy number of a hyperbolic domain $D \subset \C$ coincides with the Hardy number of its universal covering map $f_D \colon \D \to D$. A proof can be found in \cite[Lemma 2.1]{KimSugawa}. The same equality holds in the Bergman case:
\begin{lemma}
	\label{lemma:number-univcov}
	Let $D \subset \C$ be a hyperbolic domain with universal covering map $f_D \colon \D \to D$. Then, for all $\alpha > -1$, 
	$$
	\h(D) = \h(f_D), \;\;\;\b(D) = \b(f_D),\;\;\; \text{and}\;\;\; \b_{\alpha}(D) = \b_{\alpha}(f_D).
	$$
	\begin{proof}
		 Since both cases ($\b$ and $\b_{\alpha}$) are similar, let us reason with the Bergman number $\b$. First, notice that $\b(D) \leq \b(f_D)$ by definition. Therefore, it suffices to prove that $\b(f_D) \leq \b(g)$ for all holomorphic maps $g \colon \D \to D$. This is a consequence of Littlewood Subordination Theorem \cite[Theorem 1.7]{Duren-Hp} and the  expression of the weighted Bergman norm as a radial integral (see, for example,  \cite[p. 77]{Duren-Apa}).
	\end{proof}
\end{lemma}

There are also some relations between both Bergman numbers:
\begin{lemma}
	\label{lemma:ineq-bab}
{\rm (a)}	Let $f$ be a holomorphic map on $\D$. Then, for all $\alpha > -1$,
	$$
	\frac{\b_{\alpha}(f)}{\alpha+2} \leq \b(f).
	$$
{\rm (b)}	Let $D$ be a domain in $\C$. 
Then, for all $\alpha > -1$,
$$
\frac{\b_{\alpha}(D)}{\alpha+2} \leq \b(D).
$$
	\begin{proof}
	(a)	Arguing by contradiction, suppose that $\b(f) < \b_{\alpha}(f)/(\alpha+2)$. In that case, choose $p > 0$ such that $\b(f) < p/(\alpha+2) < \b_{\alpha}(f)/(\alpha+2)$. Since $\b(f) < p/(\alpha+2)$, it follows that $f \not\in A^p_{\alpha}$. But this contradicts the fact that $p < \b_{\alpha}(f)$.\\
	(b) This follows from (a) and the definitions given in (\ref{b-def}) and (\ref{eq:def-balpha}).
	\end{proof}
\end{lemma}

\begin{remark}
	The inequality in the latter result is, in general, strict. For example, consider the spaces $A^2_0$ and $A^5_2$. If follows from Theorem \ref{thm:Arevalo} that $A^5_2 \not\subset A^2_0$. Therefore, choose $f \in A^5_2 \setminus A^2_0$. Since $f \in A^5_2$, it follows that $\b(f) \geq 5/(2+2) = 5/4$. Since $f \not\in A^2_0$, we have that $\b_0(f) \leq 2$. So, $\b_0(f)/(0+2) \leq 1$. That is, $f$ is an example where $\b_0(f)/(0+2) < \b(f)$.
\end{remark}

\subsection{The Bloch space.}
Recall that the Bloch space $\mathcal{B}$ is the collection of all holomorphic maps $f \colon \D \to \C$ with
$$\sup_{z \in \D}(1-\abs{z}^2)\abs{f'(z)} < +\infty.$$
The following result is  classical. Here we use the notation $D(z,r) = \{w \in \C : \abs{w-z} < r\}$ for $z \in \C$ and $r > 0$.
\begin{theorem}
{\rm (see \cite{ACP-Bloch} and \cite[p. 35]{PommerenkeUnivFun})} 
\label{thm:Bloch-domain}
Let $D \subset \C$ be a hyperbolic domain, and consider a universal covering map $f_D \colon \D \to D$. The following are equivalent: \\
{\rm (a)} $\mathrm{Hol}(\D,D) \subset \mathcal{B}$, \\
{\rm (b)} $f_D \in \mathcal{B}$, \\
{\rm (c)} $\sup_{z \in D}R(z) < +\infty$, where $R(z) = \sup\{r > 0 : D(z,r) \subset D\}$.
\end{theorem}
Due to its importance, we say that a domain $D \subset \C$ is a Bloch domain if it satisfies the geometric property given in condition (c). Since $\mathcal{B} \subset A^p_{\alpha}$ for every $0<p<+\infty$ and every $\alpha>-1$ (see e.g. \cite[p. 44]{Duren-Apa}), the following result is obtained.
\begin{corollary}
\label{cor:Bloch-imples-binfinity}
If $D \subset \C$ is a Bloch domain, then $\b(D) = \b_{\alpha}(D) = +\infty$ for all $\alpha > -1$.
\end{corollary}


\section{Proof of Theorem \ref{thm:main-formula-Hardy}}
In this section, $D \subset \C$ denotes a Greenian domain with $0 \in D$, and $f_D \colon \D \to D$ denotes a universal covering map with $f_D(0) = 0$. We will also use the integral form $\psi_D$ of the Green function, given in \eqref{eq:integralmean}.
We first show that $\psi_D$ is a decreasing function (we use this term in the sense of non-strict monotonicity). This turns out to be a key property of $\psi_D$. 
\begin{lemma}
	\label{lemma:decreasing}
	The function $\psi_D$ is decreasing on $(0,+\infty)$.
	\begin{proof}
		Assume first that $D$ is regular for the Dirichlet problem. Then, the extended function $g_D(\cdot,0)$ is continuous and subharmonic on $\C\setminus \{0\}$. It follows (see \cite[Theorem 2.12]{HK}) that $\psi_D$ is a convex function of $\log r$, $0<r<+\infty$. Therefore, either $\psi_D$ is decreasing or there exists $r_0>0$ such that $g$ is decreasing on $(0,r_0)$ and increasing on $(r_0,\infty)$. 
		
		We will show that the latter case cannot occur: The point at $\infty$ is either a non-boundary point of $D$ (i.e., $\infty \not\in \overline{D}$ in the topology of $\C_{\infty}$) or a regular boundary point of $D$. Therefore, (see \cite[Theorem 4.4.9]{Ransford}) $\lim_{z\to\infty}g_D(z,0)=0$. So, by the Dominated Convergence Theorem, 
		$$
		\lim_{r\to\infty}\psi_D(r)=0.
		$$
		 Hence $\psi_D$ cannot be increasing on any interval $(r_0,\infty)$. We proved that the second case cannot occur and thus $\psi_D$ is decreasing.
		
		We now remove the assumption that $D$ is regular. Suppose only that $D$ is a Greenian domain. Let $D_n$, $n \in \N$, be an increasing sequence of regular domains with $0\in D_n$ for every $n \in \N$ and $\cup_{n=1}^\infty D_n =D$ (a standard construction can be found, for example, in \cite[Section XI.14]{Kellogg}). Let $g_{D_n}$ be the Green function of $D_n$ and set $\psi_n(r) :=\int_0^{2\pi} g_{D_n}(re^{i\theta},0)d\theta$. By \cite[Theorem 4.4.6]{Ransford}, $\lim_n g_{D_n}(z,0)=g_D(z,0)$, $z\in D$. The Monotone Convergence Theorem yields $\lim_{n\to\infty}\psi_n(r)=\psi_D(r)$, $r\in (0,\infty)$. By the first part of the proof, each $\psi_n$ is decreasing, and so $\psi_D$ is decreasing.
	\end{proof}
\end{lemma}

We now derive a characterization for the Hardy spaces to which a universal covering map belongs (cf. \cite[Theorem 3.1]{K-Domain}).
\begin{lemma}

	\label{lemma:Hp-integral}
	Let $0 < p < +\infty$. The following are equivalent:\\
	{\rm (a)}  $f_D \in H^p$. \\
	{\rm (b)}	$\int_0^{+\infty}r^{p-1}\psi_D(r)dr < +\infty.$\\
		{\rm (c)} $\int_D\abs{w}^{p-2}g_D(0,w)dA(w)< +\infty$. 
	\begin{proof}
		For $w\in D$, let $z_j(w),\;j=1,2,\dots$, be the preimages of $w$ by $f_D$.
		By Theorem \ref{thm:green} and a non-univalent change of variables (see \cite[Theorem 3.9]{EvansGariepy}), 
		\begin{eqnarray}
		\int_{\D}\abs{f_D(z)}^{p-2}\abs{f_D'(z)}^2\log\frac{1}{\abs{z}}dA(z) & =& \int_D\abs{w}^{p-2}\sum_j\log\frac{1}{\abs{z_j(w)}}dA(w) \nonumber \\
		& =& \int_D\abs{w}^{p-2}\sum_jg_{\D}(0,z_j(w))dA(w) \nonumber \\
		& = &\int_D\abs{w}^{p-2}g_D(0,w)dA(w) \nonumber \\
		& = &\int_0^{+\infty}r^{p-1}\psi_D(r)dr.\nonumber
		\end{eqnarray}
	 The result is, then, a direct consequence of Theorem \ref{thm:Hp}.
	\end{proof}
\end{lemma}

As a consequence of the latter characterization, we now prove one of our main results (cf. \cite[Theorem 1.2]{K-Domain}).
	\begin{proof}[Proof of Theorem \ref{thm:main-formula-Hardy}]
		The equality $\h(f_D) = \h(D)$ can be found in \cite[Lemma 2.1.(5)]{KimSugawa}.
		
		For the rest of the proof, in order to ease the notation, set
		$$p = \liminf_{r \to +\infty}\left(-\dfrac{\log \psi_D(r)}{\log r}\right).$$
		
		We start by showing that $p \leq \h(D)$. If $p = 0$, this is trivial. Suppose that $p > 0$ and choose $0 < q < p$ and $\epsilon > 0$. By the definition of $p$ there exists $R > 0$ such that
		$$-\dfrac{\log \psi_D(r)}{\log r} \geq p-\epsilon, \quad \text{for all } r > R.$$
		In that case, $\psi_D(r) \leq r^{\epsilon-p}$ for all $r > R$. Therefore,
		$$\int_0^{+\infty}r^{q-1}\psi_D(r)dr \leq \int_0^Rr^{q-1}\psi_D(r)dr + \int_R^{+\infty}r^{q-p+\epsilon-1}dr.$$
		In the right hand term, the first integral converges because $\psi_D(r)+2\pi\log(r)$ remains bounded as $r \to 0^+$, which follows the definition of the Green function $g_D$. The second integral converges if $0 < q < p-\epsilon$. As a result, it follows from Lemma \ref{lemma:Hp-integral} that $f_D \in H^{q}$ for all $0 < q < p-\epsilon$. Letting $\epsilon \to 0^+$, we deduce that $p \leq \h(D)$.

		To complete the proof, we argue by contradiction. Suppose  $p < \h(D)$. Choose some $p < q < \h(D)$. By definition, $f_D \in H^q$. Therefore, by Lemma \ref{lemma:Hp-integral},
		$$\int_0^{+\infty}t^{q-1}\psi_D(t)dt < +\infty.$$
		In particular, by Lemma \ref{lemma:decreasing}, for $r>0$,
\begin{align}
\label{eq:rq-bound}
\dfrac{r^q}{q}\psi_D(r) & = \psi_D(r)\int_0^rt^{q-1}dt \leq \int_0^rt^{q-1}\psi_D(t)dt \\
& \leq \int_0^{+\infty}t^{q-1}\psi_D(t)dt < +\infty. \nonumber
\end{align}
		This means that there exists $C > 0$ such that $r^q\psi_D(r) \leq C$ for all $r > 0$. Thus,
		$$\liminf_{r \to +\infty}\left(-\dfrac{\log \psi_D(r)}{\log r}\right) \geq q,$$
		which yields $p \geq q$, a contradiction.
	\end{proof}


\section{Proof of Theorem \ref{thm:Hardy-neq-Bergman}}
Let us commence with the proof of the first part of Theorem \ref{thm:Hardy-neq-Bergman}.
\begin{proof}[Proof of Theorem \ref{thm:Hardy-neq-Bergman}.(a)]
Assume first that $0<p<+\infty$. By \cite[Theorem 5.2]{CCZKRP}, there exists a domain $\Omega \subset \C$  such that $\h(\Omega) = p$. Set $A=\{m+in: m,n \in \Z\}$. Consider the domain $D=\Omega\setminus A$. Since $A$ is  a polar set (see \cite[Theorem 2.4]{KimSugawa}), $\h(D)=\h(\Omega)=p$. 

On the other hand, $D$ does not contain arbitrarily large disks, and therefore it is a Bloch domain; see Theorem \ref{thm:Bloch-domain}. Hence, from Corollary \ref{cor:Bloch-imples-binfinity}, we have that $\b(D)=\b_{\alpha}(D)=+\infty$, $\alpha > -1$. 

If $p=0$, we repeat the same argument as above with $\Omega = \C$.
\end{proof}

Now, for the proof of the second part of Theorem \ref{thm:Hardy-neq-Bergman}, we construct a domain $D \subset \C$ with the following three properties:
(i) $D$ is regular, (ii) $\b(D) = \b_{\alpha}(D) = +\infty$, $\alpha > -1$, and (iii) $\h(D) = 1/2$.
The domain $D$ is of the form $D = \C \setminus ((-\infty,-1] \cup C)$, where $C$ is a collection of non-trivial circular arcs (i.e., none of them is a singleton). In particular, by \cite[Theorem 4.2.2]{Ransford}, $D$ is a regular domain (notice that infinity, considered as a point on the boundary of $D$, is also regular). The arcs on $C$ will be located so that $D$ is a Bloch domain; see Theorem \ref{thm:Bloch-domain}. Once more, by Corollary \ref{cor:Bloch-imples-binfinity}, this assures that $\b(D) = \b_{\alpha}(D) = +\infty$, $\alpha > -1$, at once. The main technical work comes from showing that the length of the arcs of $C$ can be chosen small enough so that $\h(D) = 1/2$ (recall that $\h(\C \setminus (-\infty,-1]) = 1/2$).

For the construction, we will need some lemmas for harmonic measure. For instance, we will strongly use the following continuity argument as a way to obtain lower estimates for harmonic measure.
\begin{lemma}
	\label{lemma:continuity}
	Let $R > 0$, and consider the domain $\Omega = \{z \in \C : \abs{z} < R\} \setminus (-R,-1]$. Let $K \subset \Omega$ be a compact set such that $\Omega \setminus K$ is a regular domain for the Dirichlet problem. Let $\{K_{\delta}\}_{\delta \in (0,1]}$ be a family of compact subsets of $\Omega \setminus K$ with the following properties:\\
	{\rm (i)} The family is increasing. That is, $K_{\delta} \subset K_{\delta'}$ if and only if $\delta \leq \delta'$.\\
	{\rm 	(ii)}	$K^* := \cap_{\delta \in (0,1]}K_{\delta}$ is a finite set.

	For $\delta \in (0,1]$, consider the domain $\Omega_{\delta} := \Omega \setminus (K \cup K_{\delta})$, and define
	$$\omega_{\delta}(z) := \omega(z,C_R,\Omega_{\delta}), \qquad z \in \Omega_{\delta},$$
	where $C_R := \{z \in \C : \abs{z} = R\}$. Then, for every $z \in \Omega^* := \Omega \setminus (K \cup K^*)$ we have that
	$$\lim_{\delta \to 0^+}\omega_{\delta}(z) = \omega(z,C_R,\Omega \setminus K).$$
\end{lemma}
\begin{proof}
	Let $\delta > 0$, and assume that $z \in \Omega_{\delta}$. By the maximum principle we see that $\omega_{\delta_1}(z) \geq \omega_{\delta_2}(z)$ whenever $0 < \delta_1 \leq \delta_2 \leq \delta$. Therefore, since the harmonic maps $\omega_{\delta}$ are uniformly bounded by $1$, we see that the function
	\begin{equation}
		\omega^*(z) = \lim_{\delta \to 0^+}\omega_{\delta}(z)
		\end{equation}
	is well-defined and it is harmonic on $\Omega_{\delta}$. 
	
	Moreover, since $\delta > 0$ is arbitrary, we see that $\omega^*$ is well-defined and harmonic on $\Omega^*$. But, since it is bounded by above and below, the singularities of $\omega^*$ on $K^*$ are removable. Therefore, we can define $\omega^*$ on $K^*$ so that it is a harmonic map on $\Omega \setminus K$.
	
	By the boundary properties of harmonic measure given in \cite[Theorem 4.3.4]{Ransford} and the regularity of $\Omega \setminus K$ it is easy to see that 
	$$\lim_{z \to \xi}\omega^*(z) = 1, \qquad \xi \in C_R \setminus \{-R\},$$
	$$\lim_{z \to \xi}\omega^*(z) = 0, \qquad \xi \in \partial(\Omega \setminus K) \setminus C_R.$$
	Therefore, by the uniqueness property in \cite[Theorem 4.3.4]{Ransford}, 
		\begin{equation}
		\omega^*(z) = \omega(z,C_R,\Omega \setminus K), \qquad z \in \Omega \setminus K.
		\end{equation}
\end{proof}

\begin{lemma}
	\label{lemma:rate}
	Let $0 < R_1 < R$ and $S = \{z \in \C: \abs{z} = R, \, \mathrm{Re}(z) \ge 0\}$.
	There exists a constant $C= C(R_1,R) > 0$ with the following property: If $R_2>R$,  $\Omega = \{z \in \C: R_1 < \abs{z} < R_2\} \setminus (-R_2,-R_1)$,  $T = \{z \in \C: \abs{z} = R_2\}$, and $z\in S$, then
	\begin{equation}\label{rate-e}
	\omega(z,T,\Omega) \geq \dfrac{C}{\sqrt{R_2}}.
	\end{equation}
	\begin{proof}
		Assume first that $R_2<2R$. 
		Set
		\begin{eqnarray}
		\Omega_1&=& \{z \in \C: R_1 < \abs{z} < 2R\} \setminus (-2R,-R_1),
		\nonumber \\
		\Omega_2&=& \left\lbrace z \in \C: \frac{R_1}{2R} < \abs{z} < 1\right\rbrace \setminus \left(-1,-\frac{R_1}{2R}\right).
		\nonumber\end{eqnarray}
		By a polarization inequality (see Theorem \ref{polarization}.(b)), the maximum principle, and the conformal invariance of harmonic measure, for every $z\in S$,
		\begin{eqnarray}
        \label{rate-p1}
		\omega(z,T,\Omega) &\geq & \omega(iR, T,\Omega)\geq \omega(iR, \{|z|=2R\}, \Omega_1)\\
		&=&\omega(\frac{i}{2},\partial \D, \Omega_2). \nonumber
		\end{eqnarray}
		It follows that $\omega(z,T,\Omega)$ is bounded below by a positive constant depending 
		only on $R$ and $R_1$; so (\ref{rate-e}) holds.

		Suppose next that $R_2\geq 2R$. By Harnack's inequality \cite[Corollary 1.3.3]{Ransford} and the domain monotonicity of Harnack's constant \cite[Corollary 1.3.7]{Ransford}, there exists a constant $C_1>0$ depending only on $R$ and $R_1$ so that
		\begin{equation}\label{rate-p3}
		\omega(z,T,\Omega)\geq  \omega(iR, T,\Omega) \geq C_1\; \omega(R,T,\Omega),\;\;\; z\in S.
		\end{equation}
		Set
		\begin{eqnarray}
		\Omega_3&=&\C\setminus (\{|z|\leq R_1\}\cup (-\infty,-R_1)),\nonumber \\
		\Omega_4&=&\D\setminus (-1,0).\nonumber
		\end{eqnarray}
		By the maximum principle and conformal invariance,
		\begin{equation}
		\omega(R,T,\Omega)\ge\omega(R,(-\infty,-R_2],\Omega_3)=\omega\left (\frac{R_1}{R}, \left [-\frac{R_1}{R_2},0\right ],\Omega_4\right ).
		\end{equation}
		To estimate the latter harmonic measure, we use the Koebe function and a square root map. It follows easily that there exists a constant $C_2=C_2(R_1, R)$ such that
		\begin{equation}\label{rate-p5}
		\omega\left (\frac{R_1}{R}, \left [-\frac{R_1}{R_2},0\right ],\Omega_4\right )\ge \frac{C_2}{\sqrt{R_2}}.
		\end{equation}
		Combining (\ref{rate-p3})-(\ref{rate-p5}), we see that (\ref{rate-e}) is true. 
	\end{proof}
\end{lemma}

We will also use the following result dealing with symmetry.
\begin{lemma}
	\label{lemma:symmetry}
	Let $R > 1$, and consider the domain $\Omega = \{z \in \C : \abs{z} < R\} \setminus (-R,-1]$. Let $K \subset \Omega$ be a compact set which is symmetric with respect to the imaginary axis.  Assume $0\notin K$. Define $C_R = \{z \in \C : \abs{z} = R\}$ and $C_R^+ = C_R \cap \{z \in \C : \mathrm{Re}(z) > 0\}$. It holds that
	$$\omega(0,C_R^+,\Omega \setminus K) \geq \dfrac{1}{2}\omega(0,C_R,\Omega \setminus K).$$
	\begin{proof}
		The result follows from a polarization argument; we use Theorem \ref{polarization} (a) to obtain
		$$\omega(0,C_R^-,\Omega \setminus K) \leq \omega(0,C_R^+,\Omega \setminus K),$$
		from which the result follows easily.
	\end{proof}
\end{lemma}

The following lemma will be one of the main arguments in our construction.
\begin{lemma}
	\label{lemma:induction}
	Let $0 < R_1 < R$ and  $K \subset \{z \in \C : \abs{z} \leq R_1\}$ be a compact set which is symmetric with respect to the imaginary axis.  There exists a constant $C = C(R_1,R) > 0$ with the following property: If $R \leq r \leq R_2$ and
	$$\Omega_r = \{z \in \C : \abs{z} < r\} \setminus ((-r,-1] \cup K), \qquad C_r = \{z \in \C : \abs{z} = r\},$$
	then
	$$\omega(0,C_{R_2},\Omega_{R_2}) \geq \dfrac{C}{\sqrt{R_2}}\omega(0,C_R,\Omega_R).$$
	\begin{proof}
		The proof is a combination of the Markov Property for harmonic measure \eqref{markov property} and  Lemmas \ref{lemma:rate} and \ref{lemma:symmetry}.
		
		More precisely, by the Markov Property \eqref{markov property},
		\begin{eqnarray}
		\omega(0,C_{R_2},\Omega_{R_2}) & =& \int_{C_R}\omega(\alpha,C_{R_2},\Omega_{R_2})\omega(0,d\alpha,\Omega_{R})\\ &\geq & \int_{C_R^+}\omega(\alpha,C_{R_2},\Omega_{R_2})\omega(0,d\alpha,\Omega_{R}) \nonumber \\
		& \geq & \inf_{\alpha \in C_R^+}\omega(\alpha,C_{R_2},\Omega_{R_2})\omega(0,C_R^+,\Omega_{R}),\nonumber
		\end{eqnarray}
		where $C_R^+ = C_R \cap \{z \in \C : \mathrm{Re}(z) > 0\}$. From Lemma \ref{lemma:symmetry} we see that
		$$\omega(0,C_R^+,\Omega_{R}) \geq \dfrac{1}{2}\omega(0,C_R,\Omega_{R}).$$
		Let us also consider
		$$\Omega = \{z \in \C : R_1 < \abs{z} < R_2\} \setminus (-R_2,-R_1).$$
		Using the Maximum Principle we see that
		$$\inf_{\alpha \in C_R^+}\omega(\alpha,C_{R_2},\Omega_{R_2}) \geq \inf_{\alpha \in C_R^+}\omega(\alpha,C_{R_2},\Omega).$$
		From this, the proof follows immediately using Lemma \ref{lemma:rate}.
	\end{proof}
\end{lemma}

\medskip

We are now ready to present the construction in Theorem \ref{thm:Hardy-neq-Bergman}.(b). The domain $D$ is given by
$$D = D(\alpha_1,\ldots,\alpha_n,\ldots) = \C \setminus \left((-\infty,-1] \cup \bigcup_{n \in \N} C(n,\alpha_n)\right),$$
where
$$C(n,\alpha) = \bigcup_{j = 0}^{2n-1}I(n,\alpha,j), \qquad I(n,\alpha,j) = \{np(n,j)e^{i\theta} : \abs{\theta} \leq \alpha\},$$
$$p(n,j) = e^{i\pi j/n}.$$
In other words, the complement of the domain $D$ is given by the half-line $(-\infty,-1]$ and by a collection of circular arcs. More precisely, given $n \in \N$, the complement of $D$ contains $2n$ circular arcs of radius $n$ and center zero. The midpoint of these arcs are uniformly distributed following the $2n$-th roots of unity, and its angular width is controlled by the coefficient $\alpha_n$.

\begin{figure}[h]
\centering
\begin{tikzpicture}[scale=0.3]

\draw (-8.5,0) -- (-1,0);

\foreach \n in {1,...,8}{
\foreach \j in {-\n,...,\n}{
\coordinate (P) at ({\n*cos(180*\j/\n-50/\n^2)},{\n*sin(180*\j/\n-50/\n^2)});
\draw (P) arc ({180*\j/\n-50/\n^2}:{180*\j/\n+50/\n^2}:\n);
}
}
\end{tikzpicture}
\caption{Domain $D$ for some choice of the coefficients $\{\alpha_n\}$.}
\end{figure}
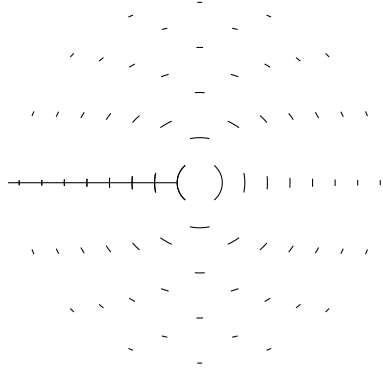

From this, we see that $D$ is clearly a regular Bloch domain. In the following proof we justify that the coefficients $\{\alpha_n\}$ can be chosen so small that $\h(D) = 1/2$.

\begin{proof}[Proof of Theorem \ref{thm:Hardy-neq-Bergman}.(b)]
	We introduce a useful notation. Given $\alpha_1, \ldots, \alpha_n > 0$ and $r > n$, we set
	$$D(r;\alpha_1,\ldots,\alpha_n) := \{z \in \C : \abs{z} < r\} \setminus \left((-r,-1] \cup \bigcup_{k = 1}^nC(k,\alpha_k)\right).$$
	As usual, we will also use the notation $C_r = \{z \in \C : \abs{z} = r\}$ for $r > 0$.
	
	The coefficients $\{\alpha_n\}$ are constructed inductively. Namely, let $A > 0$ be such that
	$$\omega(0,C_2,D(2;0)) = \omega(0,C_2,\{z \in \C : \abs{z} < 2\} \setminus (-2,-1]) = \dfrac{A}{\sqrt{2}\log(2)}.$$
	Using Lemma \ref{lemma:continuity} (with $K = \emptyset$ and $K_{\delta} = C(1,\delta)$) we can find $\alpha_1 > 0$ so that
	$$\omega(0,C_2,D(2;\alpha_1)) \geq \dfrac{A}{2\sqrt{2}\log(2)}.$$
	
	Now, assume that we have found $\alpha_1, \ldots, \alpha_n > 0$ and $n < r \leq n+1$ so that
	$$\omega(0,C_r,D(r;\alpha_1,\ldots,\alpha_n)) \geq \dfrac{A}{2\sqrt{r}\log(r)}.$$
	Using Lemma \ref{lemma:induction} we see that there exists $B = B(n,r) > 0$ such that
	\begin{equation*}
	\omega(0,C_R,D(R;\alpha_1,\ldots,\alpha_n)) \geq \dfrac{B}{\sqrt{R}}, \qquad R \geq r.
	\end{equation*}
	
	
    In particular, it is clear that
    $$\lim_{R \to \infty} \log(R)\sqrt{R}\omega(0,C_R,D(R;\alpha_1,\ldots,\alpha_n)) = +\infty.$$
    Thus, we can find $R > n+1$ such that
    $$\omega(0,C_R,D(R;\alpha_1,\ldots,\alpha_n)) \geq \dfrac{A}{\sqrt{R}\log(R)}.$$    
    In that case, let $m = \sup(\{k \in \N : k < R\})$, and notice that $m \geq n+1$. Using Lemma \ref{lemma:continuity} (here, we use
	$$K = \bigcup_{k = 1}^nC(k,\alpha_k), \qquad K_{\delta} = \bigcup_{k = n+1}^{m}C(k,\delta))$$
	we can find $\alpha_{n+1}, \ldots, \alpha_m > 0$ such that
	$$\omega(0,C_R,D(R;\alpha_1,\ldots,\alpha_m)) \geq \dfrac{A}{2\sqrt{R}\log(R)}.$$
	
	By the construction, it is clear that we can find $\{\alpha_n > 0\}$ and $\{r_{n_k} > 0\}$ with $n_k < r_{n_k} \leq n_k+1$ such that
	\begin{equation}
	\label{eq:estimate}
	\omega(0,C_{r_{n_k}},D(r_{n_k};\alpha_1,\ldots,\alpha_{n_k})) \geq \dfrac{A}{2\sqrt{r_{n_k}}\log(r_{n_k})}.
	\end{equation}
	
	Let $D := D(\alpha_1,\ldots,\alpha_n,\ldots)$. We know that $\h(D) \geq 1/2$. By \eqref{eq:estimate} and the Ess\'en-Kim-Sugawa formula (see Subsection \ref{subsec:EKS}), we conclude that $\h(D) = 1/2$.
\end{proof}
\begin{remark}
Let $p \in [1/2,+\infty)$. Modifying the geometry of the latter proof (i.e., replacing $C \setminus (-\infty,-1]$ by some angular region), one could argue that there exists a regular Bloch domain $D \subset \C$ for which $\h(D) = p$.
\end{remark}


\section{Properties of domains in the class $\mathcal{D}$}
\label{sec:class-D-lemmas}

In this section, for a domain $D$ in the class $\mathcal{D}$, we will use the notation set in Definition \ref{def:classD}: $F$ is the union of the bounded components of $\C_\infty\setminus D$ and $\Omega$ is the simply connected domain $D \cup F$. 

To begin, we notice the following fact: if $\Omega=\C$, the set $\C\setminus D$ is bounded. So, by Lemma \ref{lemma:ineq-b},
$\h(D)=\h(\Omega)=\b(D)=\b(\Omega)=0$.
It follows from Lemma \ref{lemma:ineq-bab} that $\b_\alpha(D)=0$. Since $\Omega=\C$, we also have $\b_\alpha(\Omega)=0$. 

From now on, we assume that $\Omega\subsetneq \C$.
 For $r>0$, we set $D_r=\{w\in D: |w|>r\}$ and $A_r=\{w\in D:|w|=r\}$. It follows from Definition \ref{def:classD} and the assumption  $\Omega\subsetneq \C$ that there exists $R=R(D)>0$ such that\\
(i) $F\subset \{w\in\C:|w|<R-1\}$,\\
(ii) for every $r\ge R$, $\partial D\cap\{w \in \C:|w|=r\}\neq\varnothing$,\\
(iii) for every $r\ge R$, $A_r$ is a single arc of the circle $\{w\in\C:|w|=r\}$.

In the sequel, $R$ will be this constant associated with the domain $D$. Now, we will prove several properties concerning those domains $D \in \mathcal{D}$ for which $\Omega \neq \C$.

\begin{lemma}\label{PL4}
Let $D\in \mathcal{D}$ with $\Omega \neq \C$. There exist constants $\sigma>\rho>R$ with the following properties:\\
{\rm (a)} If $a\in A_\rho$, then $\omega(a,A_R,D_R)<\frac{1}{4}$.\\
{\rm (b)} If $\zeta\in A_\rho$ and $w\in A_\sigma$, then $\rho_\Omega(\zeta,w)>\frac{\log 2}{2}$.
\end{lemma}	
\begin{proof}
(a) By a classical symmetrization theorem of Baernstein \cite[Theorem 9.5]{Hay}, for every $r>R$ and every $a\in A_r$, 
$$
\omega(a,A_R,D_R)\leq \omega(r, \{z:|z|=R\}, \Omega^*),
$$
where $\Omega^*=\C\setminus (\{z:|z|\leq R\}\cup (-\infty,-R))$. Since
$$
\lim_{r\to +\infty}\omega(r, \{z:|z|=R\}, \Omega^*)=0,
$$
we can find $\rho>R$ large enough so that for $a\in A_\rho$, $\omega(a,A_R,D_R)<\frac{1}{4}$.

\medskip

(b) Let $\rho$ be the constant we found in (a). By a classical symmetrization theorem of Weitsman \cite[Theorem 9.16]{Hay}, if $s>\rho$, $\zeta\in A_\rho$ and $w\in A_s$, then $\rho_\Omega(\zeta,w)\geq \rho_{\tilde{\Omega}}(\rho,s)$, where
$\tilde{\Omega}=\C\setminus (-\infty,-R]$. Since $$\lim_{s\to +\infty}\rho_{\tilde{\Omega}}(\rho,s)=+\infty,$$ we can find $\sigma>\rho$ large enough so that 
$$
\rho_\Omega(\zeta,w)>\frac{\log 2}{2},\;\;\;\zeta\in A_\rho,\;w\in A_\sigma.
$$
\end{proof}

In the sequel, $\sigma,\rho$ will be these constants associated with the domain $D$.

\begin{lemma}\label{PL5}
Let $D\in \mathcal{D}$ with $\Omega \neq \C$. There exists a closed arc $K\subset A_{\rho}$ with the following property: If $\zeta,w\in \Omega$, $|\zeta|\leq R$, $|w|\geq \sigma$, and $\gamma$ is the hyperbolic geodesic arc for $\Omega$ connecting $\zeta$ and $w$, then $\gamma\cap A_{\rho}\subset K$. 	
\end{lemma}	
\begin{proof}
First we note that every hyperbolic geodesic arc $\gamma$ for $\Omega$ connecting the points $\zeta,w\in \Omega$, $|\zeta|\leq R$, $|w|\geq \sigma$, has a closed subarc in $\{z\in\C:R\leq |z|\leq \sigma\}\cap \Omega$ connecting two points $\zeta_1\in A_R$, $w_1\in A_{\sigma}$. 

Next, suppose that the conclusion of the lemma is not true. Then we can find sequences $(\zeta_n)$, $(w_n)$ in $\Omega$ with the properties:\\
(a) $|\zeta_n|=R$, $|w_n|=\sigma$,\\
(b) If $\gamma_n$ is the hyperbolic geodesic arc for $\Omega$ connecting $\zeta_n$ with $w_n$, $n=1,2,\dots$, then there exists $a_n\in \gamma_n\cap A_{\rho}$ with $a_n\to a_\infty\in\partial \Omega$.

Let $f$ be a Riemann map of $\D$ onto $\Omega$. Let $\tilde{\gamma}_n$, $\tilde{\zeta}_n$, $\tilde{w}_n$, $\tilde{a}_n$, $\tilde{A}_R$, $\tilde{A}_{\rho}$, $\tilde{A}_{\sigma}$ be the images of  $\gamma_n$, $\zeta_n$, $w_n$, $a_n$, $A_R$, $A_{\rho}$, $A_{\sigma}$, respectively, under $f^{-1}$. 
By a classical theorem for conformal maps (see e.g. \cite[Proposition 3.3.3]{BCD}), the curves  $A_R$, $A_{\rho}$, $A_{\sigma}$ are crosscuts of $\D$ having six distinct endpoints on $\partial \D$. By passing to subsequences, we may assume that
\begin{equation}
\label{PL5p1}
\tilde{a}_n\to \tilde{a}_\infty\in\partial \D,\;\;\tilde{\zeta}_n\to\tilde{\zeta}_\infty \in \overline{\D},\;\;\tilde{w}_n\to\tilde{w}_\infty \in \overline{\D},\;\;\hbox{as}\;\;n\to\infty.
\end{equation}

The curves $\tilde{\gamma}_n$ are arcs of circles orthogonal to $\partial\D$ and they join the points $\tilde{\zeta}_n$, $\tilde{w}_n$. It follows that 
$\tilde{\gamma}_n\to \tilde{\gamma}_\infty$ (in the Hausdorff metric), where 
$\tilde{\gamma}_\infty$ is an arc of a circle orthogonal to $\partial\D$ joining the points $\tilde{\zeta}_\infty$, $\tilde{w}_\infty$. Since 
$\tilde{a}_n\to \tilde{a}_\infty$ and $\tilde{a}_n\in \tilde{\gamma}_n$, we have $\tilde{a}_\infty\in \tilde{\gamma}_\infty\cap \partial \D$. Therefore, either  $\tilde{\zeta}_\infty\in\partial \D$ and $\tilde{a}_\infty=\tilde{\zeta}_\infty$ or $\tilde{w}_\infty\in\partial \D$ and $\tilde{a}_\infty=\tilde{w}_\infty$. It follows that in the first case $\tilde{\zeta}_\infty$ is an endpoint of $\tilde{A}_R$ and in the second case that $\tilde{w}_\infty$ is an endpoint of $\tilde{A}_{\sigma}$. Noting that 
$\tilde{a}_\infty$ is an endpoint of $\tilde{A}_{\rho}$ and recalling that all these endpoints are distinct, we arrive at a contradiction. 
\end{proof}

\begin{lemma}\label{PL6}
Let $D\in \mathcal{D}$ with $\Omega \neq \C$. Let $a$ be the midpoint of the arc $A_{\rho}$. There exists a constant $C=C(D)$ such that for every $w\in D_{\sigma}$, $g_\Omega(a,w)\leq C\,g_{D_R}(a,w)$.		
\end{lemma}	
\begin{proof}
Let $w\in D_{\sigma}$. Let $\zeta_w\in A_R$ be such that
\begin{equation}\label{PL6p1}
g_\Omega(\zeta_w,w)=\max_{\zeta\in A_R}g_\Omega(\zeta,w).
\end{equation}
Let $a_o$ be a point on $A_{\rho}$ such that
\begin{equation}\label{PL6p2}
\omega_o:=\omega(a_o,A_R,D_R)=\max_{z\in A_{\rho}}\omega(z,A_R,D_R).
\end{equation}
Note that $\omega_o$ is a constant depending only on $D$ and that $\omega_o<1/4$, by Lemma \ref{PL4}. 
Consider the hyperbolic geodesic arc $\gamma_w$ for $\Omega$ joining $\zeta_w,w$. Let $a_w$ be a point on $A_{\rho}\cap \gamma_w$. By the Strong Markov Property \eqref{smp}, (\ref{PL6p1}), and  (\ref{PL6p2}),
\begin{eqnarray}\label{PL6p3}
g_\Omega(a_w,w)&=&g_{D_R}(a_w,w)+\int_{A_R}g_\Omega(\zeta,w)\omega(a_w,d\zeta, D_R)\\
&\leq &
g_{D_R}(a_w,w)+\omega_o\;g_\Omega(\zeta_w,w).\nonumber
\end{eqnarray}
Because of Lemma \ref{PL1},
\begin{eqnarray}
\label{PL6p4}
g_\Omega(\zeta_w,w)&\leq & 4e^{-2\rho_\Omega(\zeta_w,w)} \leq 4 e^{-2\rho_\Omega(a_w,w)} \\ &\leq & 4g_\Omega(a_w,w).\nonumber
\end{eqnarray}
Hence (\ref{PL6p3}) gives
\begin{equation}
\label{PL6p5}
g_\Omega(a_w,w)\leq C_1\,g_{D_R}(a_w,w), \;\;\;C_1=C_1(D).
\end{equation}
By Lemma \ref{PL5}, the point $a_w$ lies on a compact set $K=K(D)\subset A_{\rho}$. Therefore, by Harnack's theorem, there exist constants depending only on $D$ so that
\begin{eqnarray}
\label{PL6p6}
g_\Omega(a,w)\leq C_2 g_\Omega(a_w,w)\leq C_3 g_{D_R}(a_w,w)\leq C_4 g_{D_R}(a,w).
\end{eqnarray}
\end{proof}

\begin{lemma}\label{PL7}
Let $D\in \mathcal{D}$ with $\Omega \neq \C$ such that $0\in D$. There exists a constant $C=C(D)$ such that for every $w\in D_{\sigma}$, $g_D(0,w)\leq C\, e^{-2\rho_D(0,w)}$. 		
\end{lemma}	
\begin{proof}
Let $a$ be the midpoint of the arc $A_{\rho}$. By the domain monotonicity of the Green function, Harnack's theorem, Lemma \ref{PL6}, Lemma \ref{PL1}, and the domain monotonicity of the hyperbolic distance, there exist constants depending only on $D$ so that for every $w\in D_{\sigma}$,
\begin{eqnarray}
\label{PL7p1}
g_D(0,w)&\leq & g_\Omega(0,w)\leq C_1 g_\Omega(a,w)\leq C_2 g_{D_R}(a,w)\\ &\leq & C_3 
 e^{-2\rho_{D_R}(a,w)} \leq  C_3
 	e^{-2\rho_D(a,w)}. \nonumber
\end{eqnarray}
Using the triangle inequality $\rho_D(0,w)\leq \rho_D(0,a)+\rho_D(a,w)$, we obtain
 \begin{equation}
 \label{PL7p2}
g_D(0,w)\leq C_3 e^{2\rho_D(0,a)}	e^{-2\rho_D(0,w)}.
 \end{equation}
Since $e^{2\rho_D(0,a)}$ is a constant depending only on $D$, the proof is complete.
\end{proof}

As a result of the latter lemmas, we provide the following property.
\begin{proposition}
\label{prop:PL8}
Let $D\in \mathcal{D}$ with $\Omega \neq \C$ such that $0 \in D$. Let $f_D$ be a universal covering map of $\D$ onto $D$ with $f_D(0)=0$. For $w\in D_{\sigma}$, denote by $z_j(w), j=1,2,\dots$ the preimages of $w$ under $f_D$ with $\abs{z_1(w)}\leq \abs{z_2(w)} \leq \dots$. Then there exists a positive constant $C$ depending only on $D$ such that
$$g_D(0,w)\leq C\,g_\D(0,z_1(w)).$$
\end{proposition}
\begin{proof}
By Lemmas \ref{PL7} and \ref{PL1},
\begin{equation}
\label{PL8p1}
g_D(0,w)\leq C 	e^{-2\rho_D(0,w)}= C e^{-2\rho_\D(0,z_1(w))}\leq C \,g_\D(0,z_1(w)).
\end{equation}
\end{proof}


\section{Proof of Theorem \ref{thm:class-D}}
\label{sec:proof-D}
Let $D \subset \C$ be a Greenian domain with $0 \in D$, and let $f_D \colon \D \to D$ be a universal covering map with $f_D(0) = 0$. Assume that $D$ satisfies the following property: there exist $R = R(D) > 0$ and $C = C(D) > 0$ such that
\begin{equation}
\label{eq:condition-Green}
g_D(0,w) \leq C \min_{\substack{z \in \D \\ f(z) = w}}g_{\D}(0,z), \qquad \text{for all } w \in D \text{ with } \abs{w} > R.
\end{equation}
Notice that, by Proposition \ref{prop:PL8}, domains $D \in \mathcal{D}$ with $\Omega \neq \C$ satisfy \eqref{eq:condition-Green}. For these domains, the following result holds.

\begin{lemma}
	\label{lemma:Apa-integral}
	Let $D \subset \C$ be a domain satisfying \eqref{eq:condition-Green}. Let $0 < p <+\infty$ and $\alpha > -1$, be such that $f_D \in A^p_{\alpha}$. Then 
	\begin{equation}\label{Apae}
	\int_0^{+\infty}r^{p-1}\psi_D(r)^{\alpha+2}dr < +\infty.
	\end{equation}
	\begin{proof} 
    Let $R = R(D) > 0$ and $C = C(D) > 0$ be the constants in \eqref{eq:condition-Green}. Let $D_R = \{z \in \D : \abs{z} > R\}$. By Theorem \ref{thm:Apa}, arguing as in the proof of Lemma \ref{lemma:Hp-integral},
		\begin{eqnarray}
		+\infty & > &\int_{\D}\abs{f_D(z)}^{p-2}\abs{f_D'(z)}^2\log(1/\abs{z})^{\alpha+2}dA(z)\nonumber \\ & =& \int_D\abs{w}^{p-2}\left(\sum_j\log(1/\abs{z_j(w)})^{\alpha+2}\right)dA(w) \nonumber\\
		& =& \int_D\abs{w}^{p-2}\left(\sum_jg_{\D}(0,z_j(w))^{\alpha+2}\right)dA(w) \nonumber\\
		& \geq &\int_D\abs{w}^{p-2}g_{\D}(0,z_1(w))^{\alpha+2}dA(w),\nonumber \\
        & \geq &\int_{D_R}\abs{w}^{p-2}g_{\D}(0,z_1(w))^{\alpha+2}dA(w),\nonumber \\
        &\geq &  \frac{1}{C^{\alpha+2}}
\int_{D_R}\abs{w}^{p-2}g_D(0,w)^{\alpha+2}dA(w) \nonumber
\\
&=&
\frac{1}{C^{\alpha+2}}\int_R^{+\infty}r^{p-1}\int_0^{2\pi}g_D(0,re^{i\theta})^{\alpha+2}d\theta dr,\nonumber
\end{eqnarray}
where $\abs{z_1(w)} \leq \abs{z_2(w)} \leq \ldots$ are the preimages of $w$ by $f_D$.

	 Since $\alpha+2 > 1$, the function $t \mapsto t^{\alpha+2}$ is convex. Jensen's inequality \cite[Theorem 2.6.2]{Ransford} assures that for $r>0$,
		$$\left(\dfrac{\psi_D(r)}{2\pi}\right)^{\alpha+2} = \left(\dfrac{1}{2\pi}\int_0^{2\pi}g_D(0,re^{i\theta})d\theta\right)^{\alpha+2} \leq \dfrac{1}{2\pi}\int_0^{2\pi}g_D(0,re^{i\theta})^{\alpha+2}d\theta.$$
		It follows that
\begin{equation}\label{L4p10}
\int_R^{+\infty}r^{p-1}\psi_D(r)^{\alpha+2}dr < +\infty.
\end{equation}
By the definition of the Green function, there exists a constant $M = M(D,R) >0$ such that
$$
g_D(0,w)\leq \log\frac{1}{|w|}+M,\;\;\;w\in D,\;\;|w|\leq R.
$$
Hence
\begin{eqnarray}\label{L4p11}
\int_0^Rr^{p-1}\psi_D(r)^{\alpha+2}dr \leq  \int_0^R r^{p-1}(\log\frac{1}{r}+M)^{\alpha+2}dr <+\infty.
\end{eqnarray}	
All in all, (\ref{Apae}) follows from (\ref{L4p10}) and (\ref{L4p11}).
\end{proof}
\end{lemma}

The latter lemma assures some equalities between Hardy numbers and Bergman numbers.
\begin{lemma}
\label{lemma:equality-HB}
Let $D \subset \C$ be a Greenian domain satisfying \eqref{eq:condition-Green}. Then,
$$\h(D) = \h(f_D) = \b(D) = \b(f_D) = \dfrac{\b_{\alpha}(D)}{\alpha+2} = \dfrac{\b_{\alpha}(f_D)}{\alpha+2}, \qquad \alpha > -1.$$
\begin{proof}
	By Lemmas \ref{lemma:ineq-b}, \ref{lemma:ineq-ba}, \ref{lemma:number-univcov}, and \ref{lemma:ineq-bab}, it suffices to show that $\b(f_D) \leq \h(D)$. Notice that the inequality trivially holds if $\b(f_D) = 0$. In other case, suppose that $f_D \in A^p_{\alpha}$. By Lemma \ref{lemma:Apa-integral},
	$$\int_0^{+\infty}r^{p-1}\psi_D(r)^{\alpha+2}dr < +\infty.$$
	By Lemma \ref{lemma:decreasing}, similarly as in \eqref{eq:rq-bound}, it is possible to deduce that there exists $C > 0$ such that $r^p\psi_D(r)^{\alpha+2} \leq C$ for all $r > 0$. Then, by Theorem \ref{thm:main-formula-Hardy}, 
	$$\h(D) = \liminf_{r \to +\infty}\left(-\dfrac{\log \psi_D(r)}{\log r}\right) \geq \dfrac{p}{\alpha+2}.$$
	This yields, by definition, $\h(D) \geq \b(f_D)$.
\end{proof}
\end{lemma}
\begin{remark}
Notice that \eqref{eq:condition-Green} is trivially satisfied whenever $D$ is a simply connected domain, since Green functions are conformally invariant. Therefore, the latter result is a generalization of the result for simply connected domains due to Karafyllia and Karamanlis in \cite{KK}.
\end{remark}

We are now in position to prove Theorem \ref{thm:class-D}.
\begin{proof}[Proof of Theorem \ref{thm:class-D}]
In the beginning of Section \ref{sec:class-D-lemmas} we justified that, if $\Omega = \C$, then
$$\h(D) =\b(D)=\h(\Omega)=\b(\Omega)=\dfrac{\b_{\alpha}(D)}{\alpha+2}=\dfrac{\b_{\alpha}(\Omega)}{\alpha+2} = 0.$$
If $D$ is hyperbolic, by Lemma \ref{lemma:number-univcov}, the latter equalities imply that
$$\h(f_D) = \b(f_D) = \dfrac{\b_{\alpha}(f_D)}{\alpha+2} = 0.$$

Let us now assume that $\Omega \neq \C$. Then, $D$ is a Greenian domain. Without loss of generality, we can assume that $0 \in D$. By Proposition \ref{prop:PL8}, $D$ satisfies \eqref{eq:condition-Green}. In particular, from Lemma \ref{lemma:equality-HB} we know that
$$\h(D) = \b(D) = \b(f_D) = \dfrac{\b_{\alpha}(D)}{\alpha+2} = \dfrac{\b_{\alpha}(f_D)}{\alpha+2}, \qquad \alpha > -1.$$

Our next goal is to prove the inequality
\begin{equation}\label{Tp4}
\h(D)\leq \h(\Omega).
\end{equation}
Let $0<p<+\infty$, and let $f_D \colon \D \to D$ be a universal covering map with $f_D(0) = 0$. Assume that $f_D \in H^p$. Then Lemma \ref{lemma:Hp-integral} implies that 
\begin{equation}\label{Tp5}
\int_D\abs{w}^{p-2}g_D(0,w)dA(w)< +\infty.
\end{equation}
We will use the positive constants $R, \rho,\sigma$ associated with $D$, which we have previously defined in  Section \ref{sec:class-D-lemmas}. Let $a$ be the midpoint of the arc $A_\rho$. By Harnack's inequality, there exists a constant $C_1>0$ (depending only on $D$) such that
\begin{equation}
\label{Tp6}
g_D(a,w)\leq C_1g_D(0,w),\;\;\;w\in D_\sigma.
\end{equation}
So the inclusion $D_R\subset D$ and (\ref{Tp5}) yield
\begin{equation}\label{Tp7}
\int_{D_\sigma}\abs{w}^{p-2}g_{D_R}(a,w)dA(w)< +\infty.
\end{equation}
By Lemma \ref{PL6}, there exists a constant $C_2>0$ (depending only on $D$) such that for every $w\in D_\sigma$, $g_\Omega(a,w)\leq C_2\,g_{D_R}(a,w)$. 
So (\ref{Tp7}) gives
\begin{equation}
\label{Tp8}
\int_{D_\sigma}\abs{w}^{p-2}g_{\Omega}(a,w)dA(w)< +\infty.
\end{equation}
Using again an argument we have used in the proof of Lemma \ref{lemma:Apa-integral} (for the Green function near its pole), we infer that
\begin{equation}
\label{Tp9}
\int_{\Omega}\abs{w}^{p-2}g_{\Omega}(a,w)dA(w)< +\infty.
\end{equation}
By Harnack's inequality, this implies that
\begin{equation}
\label{Tp10}
\int_{\Omega}\abs{w}^{p-2}g_{\Omega}(0,w)dA(w)< +\infty.
\end{equation}
Therefore (see Lemma  \ref{lemma:Hp-integral}), $f_\Omega\in H^p$. 
So, we have proved that $f_D\in H^p$ implies $f_\Omega\in H^p$. It follows from the definition of the Hardy number that $\h(f_D)\leq \h(f_\Omega)$. Then, by Lemma \ref{lemma:number-univcov}, $\h(D)=\h(\Omega)$  and (\ref{Tp4}) has been proved. 

Note that, since $\Omega$ is simply connected, we have $\h(\Omega)=\b(\Omega)$.
Hence, we obtain
\begin{equation}
\label{Tp11}
\b(\Omega)\leq \b(D)=\h(D)\leq \h(\Omega)=\b(\Omega).
\end{equation}
The proof is completed by using (\ref{Tp11}) and Lemmas \ref{lemma:ineq-ba} and \ref{lemma:ineq-bab}.
\end{proof}


\section{Open problems}
\label{sec:open-problems}
A basic open problem concerning the Hardy number and the Bergman number is the following:
\begin{problem}
Find geometric-topological conditions on a domain $D\subset \C$
that are necessary and sufficient for the equality $\b(D)=\h(D)$. 
\end{problem}

Proposition \ref{prop:PL8} states that domains $D \in 
 \mathcal{D}$ with $\Omega \neq \C$ satisfy the condition \ref{eq:condition-Green}. This turns out to be fundamental in the proof of Theorem \ref{thm:class-D}. For this reason, we state the following:
\begin{problem}
Find geometric-topological conditions on a Greenian domain $D\subset \C$ that are necessary and sufficient for the condition \eqref{eq:condition-Green}, or for the equivalent inequality
\begin{equation}
g_D(z,w)\leq C(D)\,e^{-2\rho_D(z,w)},\;\;\;\;\;\;z,w\in D,\;\rho_D(z,w)\geq 1.
\end{equation} 
\end{problem}

The examples of domains we studied suggest that the Hardy and Bergman numbers depend on the geometry of the domain near infinity. Here is a related problem:
\begin{problem}
Let  $D_1, D_2 \subset \C$ be domains. Suppose that there exists $R>0$ such that
$$D_1\cap \{z:|z|>R\} = D_2\cap \{z:|z|>R\}.$$
Is it true that $\h(D_1)=\h(D_2)$? Is it true that $\b(D_1)=\b(D_2)$? For $\alpha > -1$, is it true that $\b_{\alpha}(D_1)=\b_{\alpha}(D_2)$?
\end{problem}

To end, we ask for a construction of a domain whose Bergman number is finite and different from its Hardy number.
\begin{problem}
Is there a domain $D \subset \C$ such that $0 \leq \h(D)<\b(D)<+\infty$?	
\end{problem}



\begin{thebibliography}{99}

\bibitem{ACP-Bloch}
J. M. Anderson, J. Clunie, and Ch. Pommerenke.
\newblock {\em On Bloch functions and normal functions.}
\newblock {J. Reine Angew. Math.} {\bf 270} (1974), 12--37.
MR0361090
	
	\bibitem{Arevalo}
	I. Ar\'evalo.
	\newblock {A characterization of the inclusions between mixed norm spaces.}
	\newblock {J. Math. Anal. Appl.}, {\bf 429} (2015), 942--955.
	\newblock {See also its corrigendum in {J. Math. Anal. Appl.}, {\bf 433} (2016), 1904--1905.}
	MR3342500

	

	
	\bibitem{BKK}
	D. Betsakos, C. Karafyllia, and N. Karamanlis.
	\newblock {\em Hyperbolic metric and membership of conformal maps in the Bergman space.}
	\newblock {Canad. Math. Bull.}, {\bf 64} (2021), 174--181.
	MR4243001
	
		
	\bibitem{BCD}
    F. Bracci, M. D. Contreras, and S. D\'{i}az-Madrigal, \textit{Continuous Semigroups of Holomorphic Functions in the Unit Disc}. Springer, 2020. MR4252032
	
	\bibitem{CCZKRP}
	M. D. Contreras, F. J. Cruz-Zamorano, M. Kourou, and L. Rodr\'iguez-Piazza.
	\newblock {\em On the Hardy number of Koenigs domains.}
	\newblock {Anal. Math. Phys. {\bf 14} (2024), no. 6, Paper No. 119, 21 pp.}
    MR4813211
	
	
	\bibitem{Duren-Hp}
	P. L. Duren.
	\newblock {\em Theory of $H^p$ Spaces}.
	\newblock Academic press (1970).
	MR0268655
	
	\bibitem{Duren-Apa}
	P. L. Duren and A. Schuster.
	\newblock {\em Bergman spaces.}
	\newblock {American Mathematical Society}, (2004).
	MR1758653
	
	\bibitem{Essen}
	M. Ess\'en.
	\newblock {\em On analytic functions which are in {$H\sp{p}$} for some positive  {$p$}}.
	\newblock {Ark. Mat.}, {\bf 19} (1981), 43--51.
	MR0625536
	
	\bibitem{EvansGariepy}
	L. C. Evans and R. F. Gariepy.
	\newblock {\em Measure Theory and Fine Properties of Functions.}
	\newblock {CRC Press, Boca Raton}, revised version (2015).
	MR3409135
	
	\bibitem{Hansen1}
	L. J. Hansen.
	\newblock {\em Hardy classes and ranges of functions}.
	\newblock {Michigan Math. J.}, {\bf 17} (1970), 235--248.
	MR0262512
	
	\bibitem{Hay} 
		W. K. Hayman.
		\newblock {\em Subharmonic functions. Vol. II}.	
		\newblock {London Math. Soc. Monogr.}, {\bf 20} (1989).
		MR1049148
	
	\bibitem{HK}
	W. K. Hayman and P. B. Kennedy.
	\newblock {\em Subharmonic functions. Vol. I}.
	\newblock {London Math. Soc. Monogr.}, {\bf 9} (1976).
	MR0460672

	
	\bibitem{K-HypDist}
	C. Karafyllia.
	\newblock {\em On the {H}ardy number of a domain in terms of harmonic measure and hyperbolic distance.}
	\newblock {Ark. Mat.}, {\bf 58} (2020), 307--331.
	MR4176082
	
	
	\bibitem{K-Comb2}
	C. Karafyllia.
	\newblock {\em On the {H}ardy numbers for comb domains}.
	\newblock {Ann. Fenn. Math.}, {\bf 47} (2022), 587--601.
	MR4400870
	
	\bibitem{K-Domain}
	C. Karafyllia.
	\newblock {\em The Bergman number of a plane domain}.
	\newblock {Illinois J. Math}, {\bf 67} (2023), 485--498.
	MR4644383
	
	
	\bibitem{KK}
	C. Karafyllia and N. Karamanlis.
	\newblock {\em Geometric characterizations for conformal mappings in weighted Bergman spaces.}
	\newblock {J. Anal. Math.}, {\bf 150} (2023), 303--324.
	MR4645751
	
	\bibitem{Kellogg}
	O. D. Kellogg.
	\newblock {\em Foundations of potential theory.}
	\newblock {Springer-Verlag}, (1967).
	MR0222317
	
	
	\bibitem{KimSugawa}
	Y. C. Kim and T. Sugawa.
	\newblock {\em Hardy spaces and unbounded quasidisks}.
	\newblock {Ann. Acad. Sci. Fenn. Math.}, {\bf 36} (2011), 291--300.
	MR2797697
	
	
	\bibitem{Marshall}
	D. E. Marshall.
	\newblock{\em Complex analysis}.
	\newblock{Cambridge University Press}, (2019).
	MR4321146 

	

 \bibitem{PommerenkeUnivFun}
Ch. Pommerenke.
\newblock {\em Univalent functions}.
\newblock {Studia Mathematica} {\bf 25}. Vandenhoeck \& Ruprecht (1975).
MR0507768
	
	
	\bibitem{PS}
	S. C. Port and C. J. Stone.
	\newblock {\em Brownian Motion and Classical Potential Theory}.
	\newblock Academic Press, New York, (1978).
	MR0492329
	

	
	
	\bibitem{Ransford}
	T. Ransford.
	\newblock {\em Potential Theory in the Complex Plane}.
	\newblock Cambridge University Press (1995).
	MR1334766
	
	\bibitem{Smith}
	W. Smith.
	\newblock {\em Composition operators between Bergman and Hardy spaces.}
	\newblock {Trans. Amer. Math. Soc.}, {\bf 348} (1996), 2331--2348.
	MR1357404
	
	
	\bibitem{Sol}
    A. Yu. Solynin.
	\newblock {\em Functional inequalities via polarization}.
	\newblock Algebra i Analiz {\bf 8} (1996), 148--185 (in Russian); English  transl. in St. Petersburg Math. J. {\bf 8} (1997), 1015--1038.
	MR1458141

	
	\bibitem{Yamashita}
	S. Yamashita.
	\newblock {\em Criteria for functions to be of Hardy class $H^p$.}
	\newblock {Proc. Amer. Math. Soc.}, {\bf 75} (1979), 69--72.
	MR0529215

	
\end{thebibliography}
\end{document}